\documentclass[11pt]{report}
\usepackage[cp1251]{inputenc}
\usepackage[ukrainian, russian, english]{babel}
\usepackage{amsfonts, amssymb}
\usepackage{amsmath}
\usepackage{amsthm}

\newtheorem{theorem}{Theorem}
\newtheorem{proposition}{Proposition}

\newtheorem{lemma}{Lemma}
\newtheorem{corollary}{Corollary}
\newfont{\msbm}{eufm10 at 16pt}

\begin{document}

\selectlanguage{english}
%\noindent { УДК 517.5} \vskip 1.5mm

\noindent \textbf{A.\,S.~Serdyuk} { (Institute of Mathematics of NAS of Ukraine, Ukraine)} \vskip 1.5mm

\noindent \textbf{T.\,A.~Stepanyuk} { (University of L$\mathrm{\ddot{u}}$beck, Germany; Institute of Mathematics of NAS of Ukraine, Ukraine)} \vskip 1.5mm

\noindent \textbf{ About Lebesgue inequalities on the classes of generalized Poisson integrals  }

%\maketitle

% If it is necessary include here a short version for the title of the paper and for the list of authors.
% In that case remove also the % symbol comment mark of the following line.
%\markboth{\em Short list of authors}{\em Short title}

{\small For the functions $f$, which can be represented in the form of the convolution \linebreak $f(x)=\frac{a_{0}}{2}+\frac{1}{\pi}\int\limits_{-\pi}^{\pi}\sum\limits_{k=1}^{\infty}e^{-\alpha k^{r}}\cos(kt-\frac{\beta\pi}{2})\varphi(x-t)dt$,  $\varphi\perp1$,  $\alpha>0,  \ r\in(0,1)$, $\beta\in\mathbb{R}$, we establish the Lebesgue-type inequalities of the form
\begin{equation*}
\|f-S_{n-1}(f)\|_{C}\leq e^{-\alpha n^{r}}\left(\frac{4}{\pi^{2}}\ln \frac{n^{1-r}}{\alpha r} 
+ \gamma_{n} \right) E_{n}(\varphi)_{C}.
\end{equation*}
These inequalities take place for all numbers  $n$ that are larger than some number   $n_{1}=n_{1}(\alpha,r)$, which constructively defined via parameters  $\alpha$ and $r$. We prove that there exists a function, such that the sign    "$\leq$" in given estimate can be changed for   "$=$".
}


%  Put here your AMS Classifications instead of the ones of the example with
%  the format \classification{primary}{secondary}

%\classificationps{42A10, 42A05, 41A16}{41A50}

%%%%%%%%%%%%%%%%%%%%%%%%%%%%%%%%%%%%%%%%%%%%%%%%%%%%%%%%%%%%%%%%%%%%%%%%%%%%%%%%%%%%%%
%%%%%%%%%%%%%%%%%%%%%%%%%%%%%%%%%%%%%%%%%%%%%%%%%%%%%%%%%%%%%%%%%%%%%%%%%%%%%%%%%%%%%%
%       Body of the paper
%%%%%%%%%%%%%%%%%%%%%%%%%%%%%%%%%%%%%%%%%%%%%%%%%%%%%%%%%%%%%%%%%%%%%%%%%%%%%%%%%%%%%%
%%%%%%%%%%%%%%%%%%%%%%%%%%%%%%%%%%%%%%%%%%%%%%%%%%%%%%%%%%%%%%%%%%%%%%%%%%%%%%%%%%%%%%

\section{Introduction}

Let $L_{p}$,
$1\leq p<\infty$, be the space of $2\pi$--periodic functions $f$ summable to the power $p$
on  $[0,2\pi)$, in which
the norm is given by the formula
$\|f\|_{p}=\Big(\int\limits_{0}^{2\pi}|f(t)|^{p}dt\Big)^{\frac{1}{p}}$; $L_{\infty}$ be the space of measurable and essentially bounded   $2\pi$--periodic functions  $f$ with the norm
$\|f\|_{\infty}=\mathop{\rm{ess}\sup}\limits_{t}|f(t)|$; $C$ be the space of continuous $2\pi$--periodic functions  $f$, in which the norm is specified by the equality
 ${\|f\|_{C}=\max\limits_{t}|f(t)|}$.

By $\rho_{n}(f;x)$ we denote the deviation of the function  $f$ from its partial Fourier sum of order   $n-1$:
\begin{equation*}
\rho_{n}(f;x):=f(x)-S_{n-1}(f;x),
\end{equation*}
where
\begin{equation*}
S_{n-1}(f;x)=\frac{a_{0}}{2}+\sum\limits_{k=1}^{n-1}\left(a_{k}\cos kx+b_{k}\sin kx \right),
\end{equation*}
\begin{equation*}
a_{k}=a_{k}(f)=\frac{1}{\pi}\int\limits_{-\pi}^{\pi}f(t)\cos kt dt, \ \ b_{k}=b_{k}(f)=\frac{1}{\pi}\int\limits_{-\pi}^{\pi}f(t)\sin kt dt,
\end{equation*}
and by $E_{n}(f)_{C}$ we denote the best uniform approximation of the function $f$ by elements of the subspace  $\tau_{2n-1}$ of trigonometric polynomials  $t_{n-1}(\cdot)$ of the order $n-1$:
\begin{equation*}
E_{n}(f)_{C}:=\inf\limits_{t_{n-1}\in \tau_{2n-1}}\| f-S_{n-1}(f)\|_{C}.
\end{equation*}

The norms  $\|\rho_{n}(f;\cdot)\|_{C}$ can be estimated via $E_{n}(f)_{C}$, using the Lebesgue inequality
\begin{equation}\label{LebeqIneq}
\| \rho_{n}(f; \cdot) \|_{C} \leq (1+ L_{n-1})E_{n}(f)_{C}, \ n\in\mathbb{N}.
\end{equation}
Here the  sequence of numbers 
\begin{equation*}
L_{n-1}=\frac{1}{\pi}\int\limits_{-\pi}^{\pi}|D_{n-1}(t)|dt=
\frac{2}{\pi}\int\limits_{0}^{2\pi}  \frac{|\sin (2n-1)t|}{\sin t} dt,
\end{equation*}
where
\begin{equation*}
D_{n-1}(t):=
\frac{1}{2}+\sum\limits_{k=1}^{\infty}\cos kt=\frac{\sin(n-\frac{1}{2})t}{2\sin \frac{t}{2}},
\end{equation*}
are called the Lebesgue constants  of the Fourier sums.

The asymptotic equality for Lebesgue constants $L_{n}$ was obtained in  \cite{Fejer}:
 \begin{equation*}
L_{n}=\frac{4}{\pi^{2}} \ln n+ \mathcal{O}(1), \ \ n\rightarrow\infty.
\end{equation*} 

For more exact estimates for the differences  $L_{n}-\frac{4}{\pi^{2}} \ln (n+a)$, $a>0$, as $n\in\mathbb{N}$ the reader can be referred to the works \cite{Akhiezer}--\cite{Shakirov}. 
In particular, it follows from \cite{Stepanets1} (see also \cite[p.97]{Natanson}) that
 \begin{equation*}
\left| L_{n-1}-\frac{4}{\pi^{2}} \ln n\right| <1,271, \ \  n\in\mathbb{N}.
\end{equation*}

Then, the inequality (\ref{LebeqIneq}) can be written in the form 
\begin{equation}\label{LebeqIneq0}
\| \rho_{n}(f; \cdot) \|_{C} \leq \left(\frac{4}{\pi^{2} }\ln n +R_{n}\right)E_{n}(f)_{C}, 
\end{equation}
where $|R_{n}|<2,271$.

On the whole space  $C$ the inequality (\ref{LebeqIneq0}) is asymptotically exact. At the same there exist subsets of functions from  $C$ and for elements of these subsets the inequality  (\ref{LebeqIneq0}) is not exact even by order  (see, e.g.,   \cite[p. 434]{Stepanets1989}).

In the paper \cite{Oskolkov} the following estimate was proved
 \begin{equation*}
\| \rho_{n}(f; \cdot) \|_{C} \leq K\sum\limits_{\nu=n}^{2n-1}\frac{E_{\nu}(f)_{C}}{\nu-n+1}, \ f\in C, \ \ n\rightarrow\infty,
\end{equation*}
(here $K$ is some absolute constant) and it was proved that this constant is exact by the order on the classes  $C(\varepsilon)$ with a given majorant of the best approximations
 $C(\varepsilon):=\{f\in C: \ E_{\nu}(f)_{C}\leq\varepsilon_{\nu}, \ \nu\in\mathbb{N} \}$, $\{\varepsilon_{\nu} \}_{\nu=0}^{\infty}$ is a sequence of nonnegative numbers, such that 
$\varepsilon_{\nu}  \downarrow 0$ as $\nu\rightarrow\infty$. This estimate sharpens Lebesgue classical inequality for "fast" decreasing $E_{\nu}$.

In \cite{Stepanets1989}--\cite{SerdyukMusienko2} (see also \cite{Stepanets1}) for the classes $C^{\psi}_{\beta}C$ of the functions $f\in C$, which are defined with a help of convolutions 
\begin{equation}\label{conv}
f(x)=\frac{a_{0}}{2}+\frac{1}{\pi}\int\limits_{-\pi}^{\pi}\Psi_{\beta}(x-t)\varphi(t)dt, \  \varphi\perp1, \ \ \varphi\in C, \ 
\ a_{0}\in\mathbb{R}, 
\end{equation}
with summable kernels $\Psi_{\beta}(t)$  whose Fourier series has the form
\begin{equation*}
\Psi_{\beta}(t)\sim \sum\limits_{k=1}^{\infty}\psi(k)\cos
\big(kt-\frac{\beta\pi}{2}\big), \ \ \psi(k)\geq 0, \ \ \beta\in\mathbb{R},
\end{equation*}
asymptotically best possible analogs of Lebesgue-type inequalities were found. 
In these inequalities the norms of deviations of Fourier sums   $\|\rho_{n}(f;\cdot)\|_{C}$ are expressed via the best approximations   $E_{n}(\varphi)_{C}$  of the function $\varphi$ (the function $\varphi$, which is connected with  $f$ with a help of equality  (\ref{conv}) is called  $(\psi,\beta)$--derivative of the function $f$ and is denoted by $f^{\psi}_{\beta}$).

Denote by $C^{\alpha,r}_{\beta}C, \ \alpha>0, \ r>0, $ the set of all $2\pi$--periodic functions, such that for all
$x\in\mathbb{R}$ can be represented in the form of convolution 
\begin{equation}\label{conv0}
f(x)=\frac{a_{0}}{2}+\frac{1}{\pi}\int\limits_{-\pi}^{\pi}P_{\alpha,r,\beta}(x-t)\varphi(t)dt,
\ a_{0}\in\mathbb{R}, \ \varphi\perp1, \
\end{equation}
where $\varphi\in C$, and $P_{\alpha,r,\beta}(t)$ is a generalized Poisson kernel of the form
\begin{equation*}
P_{\alpha,r,\beta}(t)=\sum\limits_{k=1}^{\infty}e^{-\alpha k^{r}}\cos
\big(kt-\frac{\beta\pi}{2}\big), \ \alpha>0, \ r>0,  \  \beta\in
    \mathbb{R}.
\end{equation*}

If $f$ and $\varphi$  are connected with a help of equality  (\ref{conv0}), then the function  $f$ in this equality is called the generalized Poisson integral of the function   $\varphi$ and is denoted by $J^{\alpha,r}_{\beta}(\varphi)$. The function $\varphi$ in the equality  (\ref{conv0}) is called the generalized derivative of the function  $f$  and is denoted by $f^{\alpha,r}_{\beta}$.

It is clear that the sets of generalized Poisson integrals 
 $C^{\alpha,r}_{\beta}C $ are subsets of the sets $C^{\psi}_{\beta}C $, if to put  $\psi(k)=e^{-\alpha k^{r}}$, $\alpha>0$, $r>0$. In this case for all  $t\in \mathbb{R}$ the equality holds $f^{\psi}_{\beta}(t)=f^{\alpha,r}_{\beta}(t)$.

It should be noticed that for any $r>0$ the classes  $C^{\alpha,r}_{\beta}C$ belong to set of infinitely differentiable
 $2\pi$--periodic functions $D^{\infty}$, i.e., $C^{\alpha,r}_{\beta}C\subset D^{\infty}$ (see, e.g., \cite[p. 128]{Stepanets1}) 
 %\cite{Stepanets_Serdyuk_Shydlich}).
For  $r\geq1$ the classes  $C^{\alpha,r}_{\beta}C$
consist of functions  $f$,  admitting a regular extension into the strip $|\mathrm{Im} \ z|\leq c, \ c>0$ in the complex
plane (see, e.g., \cite[p.~141]{Stepanets1}), i.e., are the classes of analytic functions.
For $r>1$
the classes  $C^{\alpha,r}_{\beta}C$  consist of functions  regular on the whole complex plane,
i.e., of entire functions (see, e.g., \cite[p.~131]{Stepanets1}). Besides,  it follows from the Theorem 1 in \cite{Stepanets_Serdyuk_Shydlich2009} that for any $r>0$ the embedding holds  $C^{\alpha,r}_{\beta}C\subset \mathcal{J}_{1/r}$, where $\mathcal{J}_{a}, a>0,$ are known Gevrey classes
\begin{equation*}
\mathcal{J}_{a}=\bigg\{f\in D^{\infty}: \ \sup\limits_{k\in \mathbb{N}}\Big(\frac{\|f^{(k)}\|_{C}}{(k!)^{a}}\Big)^{1/k}<\infty \bigg\}.
\end{equation*}

 In the paper of Stepanets \cite{Stepanets1989} the general results were obtained. From them, in particular, it follows that for any    $f\in C^{\alpha,r}_{\beta}C $, $r\in(0,1)$, $\alpha>0$, $\beta\in \mathbb{R}$,  for any  $n\in \mathbb{N}$ the following asymptotically best possible inequality holds  
\begin{equation}\label{StepanetsIneq}
\| \rho_{n}(f;x)\|_{C}\leq e^{-\alpha n^{r}}\left(\frac{4}{\pi^{2}}\ln n^{1-r} 
+ \mathcal{O}(1) \right) E_{n}(f^{\alpha,r}_{\beta})_{C}, \
\end{equation}
where $\mathcal{O}(1)$  is a quantity uniformly bounded with respect to  $f\in C^{\alpha,r}_{\beta}C $,  $n\in \mathbb{N}$ and $\beta\in \mathbb{R}$.

Herewith the behavior  (speed of increasing) of the quantity   $\mathcal{O}(1)$ in the inequality  (\ref{StepanetsIneq}) with respect to values of parameters   $\alpha$ and $r$ in the work   \cite{Stepanets1989}  was not considered.

In present paper we establish the asymptotically best possible Lebesgue-type inequalities for the functions  $f\in C^{\alpha,r}_{\beta}C $, in which for all $n$, starting from some number  $n_{1}=n_{1}(\alpha,r)$, an additional term is estimated by absolute constant. Herewith the number $n_{1}$ is defined constructively via parameters of the problem  (the inequality (\ref{n_1})), and an absolute constant is written in an explicit form  $20\pi^{4}$. Obtained results complement the results of the papers  \cite{SerdyukStepanyuk2019Lebesg}--\cite{SerdyukStepanyuk2018Bulleten}, and also clarify the estimate (\ref{StepanetsIneq}), which was obtained in  \cite{Stepanets1989}.

\section{Main results}

Let us formulate now the main results of the paper.

For arbitrary $\alpha>0$, $r\in(0,1)$ we denote by  $n_{1}=n_{1}(\alpha,r)$ the smallest integer $n\in\mathbb{N}$, such that
\begin{equation}\label{n_1}
 \frac{1}{\alpha r}\frac{1}{n^{r}}\Big(1+\ln \frac{\pi  n^{1-r}}{\alpha r}\Big)+\frac{\alpha r }{n^{1-r}}\leq
\frac{1}{(3\pi)^3}.
\end{equation}

\begin{theorem}\label{Theorem1}
Let $\alpha>0$, $r\in(0,1)$, $\beta\in \mathbb{R}$ and  $n\in\mathbb{N}$. Then, for any function  $f\in C^{\alpha,r}_{\beta}C $ and all $n\geq n_{1}(\alpha,r)$ the following inequality holds 
\begin{equation}\label{Theorem1_Ineq}
\|\rho_{n}(f;\cdot)\|_{C}\leq e^{-\alpha n^{r}}\left(\frac{4}{\pi^{2}}\ln \frac{n^{1-r}}{\alpha r} 
+\gamma_{n} \right) E_{n}(f^{\alpha,r}_{\beta})_{C}.
\end{equation}
Moreover, for arbitrary function  $f\in C^{\alpha,r}_{\beta}C $ one can find a function  $F(x)=F(f,n,x)$ from the set $C^{\alpha,r}_{\beta}C $, such that $E_{n}(F^{\alpha,r}_{\beta})_{C}=E_{n}(f^{\alpha,r}_{\beta})_{C}$, such that for $n\geq n_{1}(\alpha,r)$ the equality  holds
\begin{equation}\label{Theorem1_Equal}
\|\rho_{n}(F;\cdot)\|_{C}= e^{-\alpha n^{r}}\left(\frac{4}{\pi^{2}}\ln \frac{n^{1-r}}{\alpha r} 
+\gamma_{n} \right) E_{n}(f^{\alpha,r}_{\beta})_{C}.
\end{equation}
In (\ref{Theorem1_Ineq}) and (\ref{Theorem1_Equal}) for the quantity $\gamma_{n}=
\gamma_{n}(\alpha,r,\beta)$ the estimate holds $|\gamma_{n}|\leq 20\pi^{4}$.
\end{theorem}

\begin{proof}
Let  $f\in C^{\alpha,r}_{\beta}C $. Then, for arbitrary $x\in \mathbb{R}$ the following integral representation takes place 
\begin{equation}\label{repr}
\rho_{n}(f;x)=f(x)-S_{n-1}(f;x)=
\frac{1}{\pi}\int\limits_{-\pi}^{\pi}f^{\alpha,r}_{\beta}(t)P_{\alpha,r,\beta}^{(n)}(x-t)dt,
\end{equation}
where
 \begin{equation}\label{kernelN}
P_{\alpha,r,\beta}^{(n)}(t):=
\sum\limits_{k=n}^{\infty}e^{-\alpha k^{r}}\cos\Big(kt-\frac{\beta\pi}{2}\Big),  \ 0<r<1, \ \alpha>0, \ \beta\in\mathbb{R}.
\end{equation}

Whereas the function $P_{\alpha,r,\beta}^{(n)}(t)$ is orthogonal to any trigonometric polynomial  $t_{n-1}\in \tau_{2n-1}$, then because of (\ref{repr})
\begin{equation}\label{for1}
\rho_{n}(f;x)=f(x)-S_{n-1}(f;x)=
\frac{1}{\pi}\int\limits_{-\pi}^{\pi}\delta_{n}(t)P_{\alpha,r,\beta}^{(n)}(x-t)dt,
\end{equation}
where
\begin{equation}\label{delta}
\delta_{n}(x)=\delta_{n}(\alpha,r,\beta;x):=f^{\alpha,r}_{\beta}(x)-t_{n-1}(x).
\end{equation}

By $t_{n-1}^{*}\in \tau_{2n-1}$ we denote the polynomial of the best uniform approximation of the function   $f^{\alpha,r}_{\beta}$, namely, such that 
\begin{equation*}
\| f^{\alpha,r}_{\beta}-t^{*}_{n-1}\|_{C}=E_{n}(f^{\alpha,r}_{\beta})_{C}. 
\end{equation*}

Then, in view of (\ref{for1}), we have
\begin{equation}\label{for2}
\|f(\cdot)-S_{n-1}(f;\cdot)\|_{C}\leq
\frac{1}{\pi}\|P_{\alpha,r,\beta}^{(n)}\|_{1}E_{n}(f^{\alpha,r}_{\beta})_{C}.
\end{equation}

As it follows from the formula (20) of the paper   \cite{SerdyukStepanyuk2017} (see also \cite{SerdyukStepanyuk2019})
for arbitrary $r\in(0,1)$, $\alpha>0$, $\beta\in\mathbb{R}$, $1\leq s<\infty$, $\frac{1}{s}+\frac{1}{s'}=1$, $n\in\mathbb{N}$ and $n\geq n_0(\alpha,r,s')$
the relation holds
\begin{equation}\label{normKern}
\frac{1}{\pi}\|P_{\alpha,r,\beta}^{(n)} \|_{s}
=e^{-\alpha n^{r}}n^{\frac{1-r}{s'}}\bigg(\frac{\|\cos t\|_{s}}{\pi^{1+\frac{1}{s}}(\alpha r)^{\frac{1}{s'}}}\mathcal{I}_{s}\Big(\frac{ \pi n^{1-r}}{\alpha r}\Big)
+
\delta_{n,s}^{(1)}\Big(\frac{1}{(\alpha r)^{1+\frac{1}{s'}}}\mathcal{I}_{s}\Big(\frac{ \pi n^{1-r}}{\alpha r}\Big)\frac{1}{n^{r}}+\frac{1}{n^{\frac{1-r}{s'}}}\Big)\bigg),
\end{equation}
where $n_0=n_0(\alpha,r,p)$ is a smallest number  $n$, such that
\begin{equation}\label{n_p}
 \frac{1}{\alpha r}\frac{1}{n^{r}}+\frac{\alpha r \chi(p)}{n^{1-r}}\leq{\left\{\begin{array}{cc}
 \frac{1}{14},  & p=1, \\
\frac{1}{(3\pi)^3}\cdot\frac{p-1}{p}, & 1< p<\infty, \\
\frac{1}{(3\pi)^3}, & p=\infty, \
  \end{array} \right.}
\end{equation}
where $\chi(p)=p$ for $1 \leq p<\infty$ and $\chi(p)=1$ for $p=\infty$ and
\begin{equation}\label{norm_j}
  \mathcal{I}_{s}(\upsilon):=
   {\left\{\begin{array}{cc}
\bigg(\int\limits_{0}^{\upsilon}\frac{1}{(\sqrt{t^{2}+1})^{s}}dt
\bigg)^{\frac{1}{s}}, & 1\leq s<\infty, \\
\mathop{\rm{ess}\sup}\limits_{t\in[0,\upsilon]}|\frac{1}{\sqrt{t^{2}+1}}|=1, \ & s=\infty, \
  \end{array} \right.}
\end{equation}
and for the quantity ${\delta_{n,s}^{(1)}=\delta_{n,s}^{(1)}(\alpha,r,\beta)}$ the following estimate holds  ${|\delta_{n,s}^{(1)}|\leq(14\pi)^{2}}$.

Putting in the formula \eqref{normKern} $s=1$, we get that for $r\in(0,1)$, $\alpha>0$ and  $\beta\in \mathbb{R}$, $n\in\mathbb{N}$ and  $n\geq n_0(\alpha,r,\infty)$ the relation takes place
\begin{equation}\label{normKernL1}
\frac{1}{\pi}\|P_{\alpha,r,\beta}^{(n)} \|_{1}
=e^{-\alpha n^{r}}\bigg(\frac{4}{\pi^{2}}\mathcal{I}_{1}\Big(\frac{ \pi n^{1-r}}{\alpha r}\Big)
+
\delta_{n,1}^{(1)}\Big(\frac{1}{\alpha r}\mathcal{I}_{1}\Big(\frac{ \pi n^{1-r}}{\alpha r}\Big)\frac{1}{n^{r}}+1\Big)\bigg).
\end{equation}
According to the formula (112) of the work  \cite{SerdyukStepanyuk2019} 
\begin{equation}\label{form112}
\mathcal{I}_{1}\Big(\frac{ \pi n^{1-r}}{\alpha r}\Big)=
\int\limits_{0}^{\frac{ \pi n^{1-r}}{\alpha r}}\frac{dt}{\sqrt{t^{2}+1}}
=\ln \frac{ \pi n^{1-r}}{\alpha r}+
\Theta_{\alpha,r,n},
\end{equation}
where  $0<\Theta_{\alpha,r,n}<1$. It is easy to show that for $n\geq n_{1}(\alpha,r)$ the following inequality holds
\begin{align*}
&\frac{4}{\pi^{2}} \left(\ln \pi+ \Theta_{\alpha,r,n} \right) +|\delta_{n,1}^{(1)}|
\left(\frac{1}{\alpha rn^{r}} \ln \frac{ \pi n^{1-r}}{\alpha r} +
\frac{\Theta_{\alpha,r,n}}{\alpha rn^{r}}+1 \right)
\notag \\
&<\frac{4}{\pi^{2}} \left(\ln \pi+1\right)
+(14\pi)^{2}\left( \frac{1}{(3\pi)^{3}}+1 \right)<1938,
\end{align*}
then formulas \eqref{normKernL1} and \eqref{form112} imply that for $n\geq n_{1}(\alpha,r)$
\begin{equation}\label{normL1}
\| P_{\alpha,r,\beta}^{(n)}\|_{1}=e^{-\alpha n^{r}}\left(\frac{4}{\pi^{2}}\ln \frac{n^{1-r}}{\alpha r}+\gamma_{n}^{*} \right),
\end{equation}
where for the quantity  $\gamma_{n}^{*}=\gamma_{n}^{*}(\alpha, r, \beta)$ the estimate is true $|\gamma_{n}^{*}|<1938$. The inequalities \eqref{for2} and \eqref{normL1} prove the truth of    \eqref{Theorem1_Ineq}.

To prove the second part of Theorem 1 it is enough to show that for any function $\varphi\in C$ one can construct a function $\Phi(\cdot)=\Phi(\varphi, \cdot)\in C$, such that  $E_{n}(\Phi)_{C}=E_{n}(\varphi)_{C}$ and for any $n\geq n_{1}(\alpha,r)$ the equality holds 
\begin{equation}\label{for115}
\frac{1}{\pi}\left|\int\limits_{-\pi}^{\pi}\Phi(t) P_{\alpha,r,\beta}^{(n)} (0-t)dt\right|
=e^{-\alpha n^{r}}\left(\frac{4}{\pi^{2}}\ln \frac{n^{1-r}}{\alpha r}+\gamma_{n} \right)E_{n}(\varphi)_{C},
\end{equation}
where $|\gamma_{n}|<20\pi^{4}$.

In this case for arbitrary function $f\in C^{\alpha, r}_{\beta}C$ there exists a function  $\Phi(\cdot)=\Phi(f^{\alpha,r}_{\beta}, \cdot)$, such that 
$E_{n}(\Phi)_{C}=E_{n}(f^{\alpha,r}_{\beta})_{C}$ and for $n\geq n_{1}(\alpha,r)$ 
formula \eqref{for115} is true, where as function $\varphi$ we take the function $f^{\alpha,r}_{\beta}$. Let us assume $F(\cdot)=J^{\alpha,r}_{\beta}(\Phi(\cdot)-\frac{a_{0}}{2})$, where
$a_{0}=a_{0}(\Phi)=\frac{1}{\pi}\int\limits_{-\pi}^{\pi}\Phi(t)dt$. The function $F$ is the function,  which we have looked for, because $F\in C^{\alpha, r}_{\beta}C$, and $E_{n}(F^{\alpha, r}_{\beta})_{C}=E_{n}(\Phi-\frac{a_{0}}{2})_{C}=E_{n}(\Phi)_{C}=E_{n}(f^{\alpha, r}_{\beta})_{C} $ and moreover, formulas \eqref{Theorem1_Ineq}, \eqref{repr}, \eqref{for1} and
\eqref{for115} yield \eqref{Theorem1_Equal}.

To prove \eqref{for115}  we need more detailed information about the character of oscillation of the kernel   $P_{\alpha, r,\beta}^{(n)}(t)$.
We denote by $n_{*}=n_{*}(\alpha,r)$ the smallest number, for which the inequality holds
\begin{equation}\label{nStar}
\frac{1}{\alpha r n^{r}}+\frac{\alpha r}{n^{r-1}}<\frac{117}{784\pi^{2}}.
\end{equation}

\begin{lemma}\label{Lemma1}
Let $\alpha>0$, $r\in(0,1)$, $\beta\in\mathbb{R}$ and $n\in\mathbb{N}$. For $n\geq n_{*}$ the function  $P_{\alpha, r,\beta}^{(n)}(t)$ has exactly  $2n$  simple zeros $z_{k}$  on the period $[0,2\pi)$, where the function $P_{\alpha, r,\beta}^{(n)}(t)$ takes values with alternating signs.
\end{lemma}
\begin{proof}
According to formulas (44) and (47) of the work \cite{SerdyukStepanyuk2019} we can write
\begin{align*}
&
P_{\alpha, r,\beta}^{(n)}(t)=g_{\alpha,r,n}(t)\cos\left(nt-\frac{\beta\pi}{2}\right)+
h_{\alpha,r,n}(t)\sin\left(nt-\frac{\beta\pi}{2}\right)
\notag \\
&=\sqrt{g_{\alpha,r,n}^{2}(t) +h_{\alpha,r,n}^{2}(t) }  \cos\left(nt-\frac{\beta\pi}{2}-
\mathrm{arctg} \frac{h_{\alpha,r,n}(t)}{g_{\alpha,r,n}(t)}\right) 
\end{align*}
\begin{equation}\label{form115}
= \sqrt{g_{\alpha,r,n}^{2}(t) +h_{\alpha,r,n}^{2}(t) }  \cos(n\cdot y(t)),
\end{equation}

where
\begin{equation}\label{g}
g_{\alpha,r,n}(t)=
\sum\limits_{k=0}^{\infty}e^{-\alpha(k+n)^{r}}\cos kt,
\end{equation}
\begin{equation}\label{h}
h_{\alpha,r,n}(t)=
\sum\limits_{k=0}^{\infty}e^{-\alpha(k+n)^{r}}\sin kt,
\end{equation}
\begin{equation}\label{funcY}
y(t)=y(\alpha,r,n;t)=
t-\frac{\beta\pi }{2n}-\frac{1}{n}\mathrm{arctg}\frac{h_{\alpha,r,n}(t)}{g_{\alpha,r,n}(t)}.
\end{equation}
Lemma 1 will be proved, if one can show that for $n\geq n_{*}$ the function $y(t)$ of the form \eqref{funcY} increasing on $[0,2\pi]$ from a value $y(0)=-\frac{\beta\pi}{2}$ to  a value $y(2\pi)=2\pi-\frac{\beta\pi}{2}$. In this case the function  $\cos(n\cdot y(t))$, and also the function  $P_{\alpha, r,\beta}^{(n)}(t)$ (taking into account \eqref{form115} and also the strict inequality
$\sqrt{g_{\alpha,r,n}^{2}(t) +h_{\alpha,r,n}^{2}(t) } >0$ (see  (47) from \cite{SerdyukStepanyuk2019})) have on $[0,2\pi)$ exactly $2n$ simple zeros $z_{k}$ of the form 
\begin{equation}\label{Zeros}
z_{k}=y^{-1}\left(\frac{\frac{\pi}{2}+k\pi}{n} \right),  \ k=0,...,2n-1,
\end{equation}
 where $y^{-1}(\cdot)$ is inverse function  to  $y(\cdot)$.  In points $z_{k}$ the function $\cos(n y(t))$
(and also the function $P_{\alpha, r,\beta}^{(n)}(t)$) takes values with alternating signs.

Let us consider the derivative of the function $y(t)$:
\begin{equation}\label{form116}
y'(t)=
1-\frac{1}{n}\frac{\left(\frac{h_{\alpha,r,n}(t)}{g_{\alpha,r,n}(t)} \right)'  }{1+ \frac{h_{\alpha,r,n}(t)^{2}}{g_{\alpha,r,n}(t)^{2}}} 
=1+\frac{1}{n}
\frac{-h'_{\alpha,r,n}(t)g_{\alpha,r,n}(t)+h_{\alpha,r,n}(t)g'_{\alpha,r,n}(t)}{g_{\alpha,r,n}(t)^{2}+h_{\alpha,r,n}(t)^{2}}.
\end{equation}
Let us estimate the absolute value of the last term in formula  \eqref{form116}
\begin{align}\label{form117}
&\frac{1}{n}
\left|
\frac{-h'_{\alpha,r,n}(t)g_{\alpha,r,n}(t)+h_{\alpha,r,n}(t)g'_{\alpha,r,n}(t)}{g_{\alpha,r,n}(t)^{2}+h_{\alpha,r,n}(t)^{2}}\right| \notag \\
&=\frac{1}{n}\frac{\sqrt{g'_{\alpha,r,n}(t)^{2}+h'_{\alpha,r,n}(t)^{2}}}{\sqrt{g_{\alpha,r,n}(t)^{2}+h_{\alpha,r,n}(t)^{2}}} \left|
\frac{-h'_{\alpha,r,n}(t)}{\sqrt{g'_{\alpha,r,n}(t)^{2}+h'_{\alpha,r,n}(t)^{2}}}
\frac{g_{\alpha,r,n}(t)}{\sqrt{g_{\alpha,r,n}(t)^{2}+h_{\alpha,r,n}(t)^{2}}} \right.
\notag \\
&\left.+
\frac{h_{\alpha,r,n}(t)}{\sqrt{g_{\alpha,r,n}(t)^{2}+h_{\alpha,r,n}(t)^{2}}}
\frac{g'_{\alpha,r,n}(t)}{\sqrt{g'_{\alpha,r,n}(t)^{2}+h'_{\alpha,r,n}(t)^{2}}}
\right| \notag \\
&\leq
\frac{1}{n}\frac{\sqrt{g'_{\alpha,r,n}(t)^{2}+h'_{\alpha,r,n}(t)^{2}}}{\sqrt{g_{\alpha,r,n}(t)^{2}+h_{\alpha,r,n}(t)^{2}}} \leq \frac{M_{n}(\alpha,r)}{n},
\end{align}
where
\begin{equation}\label{Mn}
M_{n}(\alpha,r):=\sup\limits_{t\in \mathbb{R}}\frac{\sqrt{g'_{\alpha,r,n}(t)^{2}+h'_{\alpha,r,n}(t)^{2}}}{\sqrt{g_{\alpha,r,n}(t)^{2}+h_{\alpha,r,n}(t)^{2}}}.
\end{equation}
From formula  (99) from \cite{SerdyukStepanyuk2019} we have 
\begin{equation*}
M_{n}\leq \frac{784\pi^{2}}{117} \left(\frac{n^{1-r}}{\alpha r}+\alpha rn^{r} \right).
\end{equation*}
This and inequality \eqref{nStar} yield that  $y'(t)>0$ for $n\geq n_{*}(\alpha,r)$,  so the function  $y(t)$ strictly increasing.
 Lemma~\ref{Lemma1} is proved.
\end{proof}
Let us prove now the estimate \eqref{for115}. Let $\varphi\in C$. Denote by $\Phi_{\delta}(t)$ the $2\pi$--periodic function, which coincides with the function 
\begin{equation*}
\Phi_{0}(t)=E_{n}(\varphi)_{C} \mathrm{sign}P_{\alpha, r,\beta}^{(n)}(-t)
\end{equation*}
everywhere, except   $\delta$--neighborhoods ($\delta<\frac{1}{2}\min\limits_{k\in\mathbb{Z}}\{ z_{k+1}-z_{k}\}$) of points $z_{k}$, where it is linear function and its graph connects the points with coordinates  $(z_{k}-\delta, \Phi_{0}(z_{k}-\delta))$ and
$(z_{k}+\delta, \Phi_{0}(z_{k}+\delta))$.

The function $\Phi_{\delta}(\cdot)$ is continuous. As the condition \eqref{n_1} is more strong than the condition  \eqref{nStar}, then $n_{1}(\alpha,r)\geq n_{*}(\alpha,r)$. On the basis of Lemma 1 for $n\geq n_{1}(\alpha,r)$ the function  $\Phi_{\delta}$ has on  $[0,2\pi)$ exactly  $2n$ zeros $z_{k}$ of the form \eqref{Zeros}, where it takes values with alternating signs, and in the middle of each interval  $(z_{k}, z_{k+1})$  it 
takes the
maximum absolute values with alternating signs
   $\pm E_{n}(\varphi)_{C}$. Then, by Chebyshev theorem about alternance, the polynomial $t_{n-1}^{*}$  of the best approximation of the function  $\Phi_{\delta}$ in  the uniform metric will be identically equal to zero and  $E_{n}(\Phi_{\delta})_{C}= E_{n}(\varphi)_{C}$.  Therefore
\begin{equation}\label{for118}
\frac{1}{\pi}\int\limits_{-\pi}^{\pi}\Phi_{\delta}(t) P_{\alpha,r,\beta}^{(n)} (0-t)dt
=
\frac{1}{\pi}\int\limits_{-\pi}^{\pi}\Phi_{0}(t) P_{\alpha,r,\beta}^{(n)} (-t)dt +R_{n}(\delta),
\end{equation}
where
\begin{equation}\label{RnDelta}
R_{n}(\delta)=
R_{n}(\alpha, r, \beta, \delta)=
\frac{1}{\pi}\int\limits_{-\pi}^{\pi}(\Phi_{\delta}(t)-\Phi_{0}(t)) P_{\alpha,r,\beta}^{(n)} (-t)dt.
\end{equation}
As,
\begin{align*}
&\frac{1}{\pi}\int\limits_{-\pi}^{\pi}\Phi_{0}(t) P_{\alpha,r,\beta}^{(n)} (-t)dt =
E_{n}(\varphi)_{C}
\int\limits_{-\pi}^{\pi}\mathrm{sign} P_{\alpha,r,\beta}^{(n)} (-t) P_{\alpha,r,\beta}^{(n)} (-t)dt 
\notag \\
&
=E_{n}(\varphi)_{C}
\int\limits_{-\pi}^{\pi}|P_{\alpha,r,\beta}^{(n)} (-t)|dt =
E_{n}(\varphi)_{C}
\int\limits_{-\pi}^{\pi}|P_{\alpha,r,\beta}^{(n)} (t)|dt=
E_{n}(\varphi)_{C}\| P_{\alpha,r,\beta}^{(n)} \|_{1},
\end{align*}
then on the basis of  \eqref{for118} and \eqref{normL1} we have that
\begin{equation}\label{for120}
\frac{1}{\pi}\int\limits_{-\pi}^{\pi}\Phi_{\delta}(t) P_{\alpha,r,\beta}^{(n)} (0-t)dt
=
e^{-\alpha n^{r}}\left(\frac{4}{\pi^{2}}\ln \frac{n^{1-r}}{\alpha r}+\gamma_{n}^{*} \right)E_{n}(\varphi)_{C}
 +R_{n}(\delta),
\end{equation}
where $|  \gamma_{n}^{*}|<1938$.

Let us choose $\delta$ small enough, that the following inequality holds
\begin{equation}\label{IneqDelta}
\delta<\frac{13\pi(10\pi^{4}-969)}{14}\frac{\alpha r n^{r}}{n^{2}}.
\end{equation}
For values $\delta$, which satisfy the condition  \eqref{IneqDelta}, for
$n\geq n_{1}(\alpha,r)$ the following estimate holds
\begin{equation}\label{IneqRN}
|R_{n}(\delta)|<(20\pi^{4}-1938) e^{-\alpha n^{r}}E_{n}(\varphi)_{C}.
\end{equation}
Indeed, according to  \eqref{RnDelta}
\begin{equation}\label{for121}
 | R_{n} (\delta)| \leq \frac{1}{\pi}
\| P_{\alpha,r,\beta}^{(n)}\|_{C} \int\limits_{-\pi}^{\pi}
|\Phi_{\delta}(t)-\Phi_{0}(t)|dt \leq
\frac{1}{\pi}\sum\limits_{k=0}^{\infty}e^{-\alpha (k+n)^{r}}
\int\limits_{-\pi}^{\pi}
|\Phi_{\delta}(t)-\Phi_{0}(t)|dt
\end{equation}
and, as follows from the formula  (91) of the work \cite{SerdyukStepanyuk2019}, for $n\geq n_{1}(\alpha,r)$, we derive the inequality
\begin{equation}\label{for122}
\sum\limits_{k=0}^{\infty}e^{-\alpha (k+n)^{r}}<
\frac{14}{13} e^{-\alpha n^{r}}\frac{n^{1-r}}{\alpha r}.
\end{equation}
Then,  whereas according to definitions of the functions   $\Phi_{\delta}$ and $\Phi_{0}$, the following equality holds
\begin{equation*}
 \int\limits_{-\pi}^{\pi}
|\Phi_{\delta}(t)-\Phi_{0}(t)|dt 
=E_{n}(\varphi)_{C} \sum\limits_{k=0}^{2n-1}\int\limits_{z_{k}-\delta}^{z_{k}+\delta}
\frac{|-t+z_{k}-\delta |}{2\delta}dt
=2n \delta E_{n}(\varphi)_{C},
\end{equation*}
then from \eqref{IneqDelta}, \eqref{for121} and \eqref{for122} we obtain the inequalities 
\begin{equation*}
|R_{n}(\delta)|< \frac{28}{13\pi}e^{-\alpha n^{r}}\frac{n^{2-r}}{\alpha r} \delta 
E_{n}(\varphi)_{C}
<2(10\pi^{4} -969)  e^{-\alpha n^{r}} E_{n}(\varphi)_{C}.
\end{equation*}
This  proves \eqref{IneqRN}.

Hence, let us put $\Phi(t)=\Phi_{\delta}(t)$, choosing some $\delta$, such that $\delta<\frac{1}{2}\min\limits_{k\in \mathbb{Z}} \{ z_{k+1}-z_{k} \}$. It should be noticed that for such choice of $\delta$,   the condition \eqref{IneqDelta} is satisfied. Then, because of \eqref{for120} and \eqref{IneqRN} for $n\geq n_{1}(\alpha,r)$ for the function  $\Phi(t)$ the estimate \eqref{for115} is true. Theorem 1 is proved.
\end{proof}

Notice that the statement of Lemma 1 takes place not only for $n\geq n_{*}$, but  for all  $n\in \mathbb{N}$. To be sure in it, we use the following statement of the work     \cite{Tveritin}.

\begin{proposition}\label{Theorem2}
Let the coefficients  $a_{k}$ of the trigonometric series
\begin{equation}\label{Theorem2_Cond1}
\sum\limits_{k=n}^{\infty}a_{k} \sin(kx+\gamma), \ \ \gamma\in\left(0, \frac{\pi}{2}\right], \ \ n\in\mathbb{N},
\end{equation}
satisfy the conditions
\begin{equation}\label{Theorem2_Cond2}
\Delta_{m}a_{k}:= a_{k}-ma_{k+1}+\frac{m(m-1)}{1\cdot 2}a_{k+2}-...+(-1)^{m}a_{k+m}>0, \ \ m\in\mathbb{Z}_{+}, \ k\in\mathbb{N},
\end{equation}
\begin{equation}\label{Theorem2_Cond3}
\lim\limits_{k\rightarrow \infty}a_{k}=0,
\end{equation}
\begin{equation}\label{Theorem2_Cond4}
\sum\limits_{k=n}^{\infty} a_{k}<\infty.
\end{equation}
Then the function-sum \eqref{Theorem2_Cond1} has exactly $2n$ simple zeros  in the interval  $(0,2\pi)$, which  are alternately  located inside of respective intervals  
\begin{equation*}
\left( \frac{\pi-2\gamma}{2n-1}, \frac{\pi-\gamma}{n} \right), \  
\left( \frac{3\pi-2\gamma}{2n-1}, \frac{2\pi-\gamma}{n} \right), ...,
\left( \frac{(2n-1)\pi-2\gamma}{2n-1}, \frac{n\pi-\gamma}{n} \right), \   
\end{equation*}
\begin{equation*}
\left( \frac{(n+1)\pi-\gamma}{n}, \frac{(2n+1)\pi-2\gamma}{2n-1} \right), ...,
\left( \frac{(2n-1)\pi-\gamma}{n}, \frac{(4n-3)\pi-2\gamma}{2n-1} \right), 
\end{equation*}
\begin{equation*}
\left( \frac{2n\pi-\gamma}{2n-1}, 2\pi \right).  
\end{equation*}
\end{proposition}

The condition \eqref{Theorem2_Cond2} of  Proposition~\ref{Theorem2} can be written in the form 
\begin{equation}\label{Theorem2_Cond2_2}
(-1)^{m}\Delta^{m}a_{k}>0, \ \ k\in\mathbb{Z}_{+}, \ m\in\mathbb{N},
\end{equation}
(so called condition of absolutely monotonicity of the sequence  $a_{k}$), where the difference operator  $\Delta^{m}$ is defined by induction with a help of equalities
\begin{align*}
& \Delta^{0}a_{k}=a_{k}, \ \ \ \Delta^{1} a_{k}=a_{k+1}-a_{k}, \ \ \ \Delta^{2} a_{k}=\Delta^{1}(\Delta^{1} a_{k}),...,  \notag \\
&\Delta^{m} a_{k}=\Delta^{1}(\Delta^{m-1} a_{k})=\sum\limits_{v=0}^{m} \binom {m}v (-1)^{m+v} a_{k+v}.
\end{align*}

To apply Proposition~\ref{Theorem2} to the sequence $a_{k}=e^{-\alpha k^{r}}$, $\alpha>0$, $r\in(0,1)$ it is enough to be sure that   \eqref{Theorem2_Cond2_2} holds, because the verification of \eqref{Theorem2_Cond3}
and  \eqref{Theorem2_Cond4} is trivial.
 For $\alpha>0$, $r\in(0,1)$ the function $\psi(t)=\psi(\alpha,r,t)=e^{-\alpha t^{r}}$ is absolutely monotonic, namely, the condition holds
\begin{equation*}
(-1)^{m}\psi^{(m)}(t)>0, \ \ m\in\mathbb{N}, \ t>0.
\end{equation*}
It follows from the fact that the function  $\psi(t)$ is a superposition of absolutely  monotonic function $exp(-t)$ and positive function  $g(t)=g(\alpha,r,t)=\alpha t^{r}$, $\alpha>0$,
 $r\in(0,1)$, which has absolutely  monotonic derivative   (see \cite[Ch. 4, §4 ]{Feller}). Hence, the sequence $a_{k}=e^{-\alpha k^{r}}$ satisfies the condition   \eqref{Theorem2_Cond2}--\eqref{Theorem2_Cond4} of  Proposition~\ref{Theorem2}, and herefrom the following statement holds.

\begin{corollary} Let $\alpha>0$, $r\in(0,1)$, $\beta\in[0, 1)$ and $n\in\mathbb{N}$. Then, on $[0,2\pi)$ the function  $P_{\alpha,r,\beta}^{(n)} (t)$ has exactly  $2n$ simple zeros, which are alternately located inside of respective intervals 
\begin{equation*}
\left( \frac{\beta\pi}{2n-1}, \frac{\pi-(1-\beta)\frac{\pi}{2}}{n} \right), \  
\left( \frac{(2+\beta)\pi}{2n-1}, \frac{2\pi-(1-\beta)\frac{\pi}{2}}{n} \right), ..., 
\end{equation*}
\begin{equation*}
\left( \frac{(2n-2+\beta)\pi}{2n-1}, \frac{n\pi-(1-\beta)\frac{\pi}{2}}{n} \right), \  
\end{equation*}
\begin{equation*}
\left( \frac{(n+1)\pi-(1-\beta)\frac{\pi}{2}}{n}, \frac{(2n+\beta)\pi}{2n-1} \right), ...,
\left( \frac{(2n-1)\pi-(1-\beta)\frac{\pi}{2}}{n}, \frac{(4(n-1)+\beta)\pi}{2n-1} \right), 
\end{equation*}
\begin{equation*}
\left( \frac{2n\pi- (1-\beta)\frac{\pi}{2}}{2n-1}, 2\pi \right).  
\end{equation*}
\end{corollary}

In the same way as it was done in the works  \cite{Stepanets1989}--\cite{SerdyukMusienko} we consider the classes   $C^{\alpha,r}_{\beta}C(\varepsilon)$ of $2\pi$--periodic functions $f$  of the form \eqref{conv}, where $\varphi=f^{\alpha,r}_{\beta}$ belongs to the class $C(\varepsilon)$, where as earlier $\varepsilon=\{\varepsilon_\nu\}_{\nu=0}^{\infty}  $ is monotonically decreasing to zero the sequence of nonnegative numbers.

The following statement gives an example that the  inequality  \eqref{Theorem1_Ineq} is  best possible
 not only on the set  $C^{\alpha,r}_{\beta}C$, but also on such important subsets $C^{\alpha,r}_{\beta}C(\varepsilon)$.  of the set  $C^{\alpha,r}_{\beta}C$.
 
 \begin{theorem}\label{Theorem3}
 Let $\alpha>0$, $r\in(0,1)$, $\beta\in\mathbb{R}$ and $\varepsilon=\{\varepsilon_\nu\}_{\nu=0}^{\infty}  $ is an arbitrary monotonically decreasing to zero the sequence of nonnegative real numbers. Then, for arbitrary class $C^{\alpha,r}_{\beta}C(\varepsilon)$ and all numbers $n\geq n_{1}(\alpha,r)$ the equalities hold
 \begin{equation}\label{Theorem3Equality}
\mathcal{E}_{n}(C^{\alpha,r}_{\beta}C(\varepsilon))_{C}=\sup\limits_{f\in C^{\alpha,r}_{\beta}C(\varepsilon)}
\|f(\cdot)-S_{n-1}(f,\cdot)\|_{C}=
e^{-\alpha n^{r}} \left(\frac{4}{\pi^{2}}\ln \frac{n^{1-r}}{\alpha r}+\gamma_{n} \right)\varepsilon_{n},
\end{equation}
where $|\gamma_{n}|\leq 20\pi^{4}$.
 \end{theorem}
 \begin{proof}
Let  $f\in C^{\alpha,r}_{\beta}C(\varepsilon)$. Then, the function $\varphi=f^{\alpha,r}_{\beta}$ is continuous and  $E_{n}(f^{\alpha,r}_{\beta})_{C}\leq \varepsilon_{n}$. Then, taking into account \eqref{Theorem1_Ineq}, we obtain that for $n\geq n_{1}(\alpha,r)$
 \begin{equation}\label{for123}
\| \rho_{n}(f;\cdot)\|_{C}\leq e^{-\alpha n^{r}} \left(\frac{4}{\pi^{2}}\ln \frac{n^{1-r}}{\alpha r}+\gamma_{n} \right)\varepsilon_{n} \ \ \forall f\in C^{\alpha,r}_{\beta}C(\varepsilon),
\end{equation}
where $|\gamma_{n}|\leq 20\pi^{4}$.
 
 On other hand, from Theorem~\ref{Theorem1} it follows that for the function   $F(x)$, which is constructed for the fucntion  $\varphi=f^{\alpha,r}_{\beta}\in C(\varepsilon))$, and such that $E_{n}(f^{\alpha,r}_{\beta})=\varepsilon_{n}$, the inequality \eqref{for123} becomes an equality for   $n\geq n_{1}(\alpha,r)$. Herefrom we get \eqref{Theorem3Equality}. Theorem~\ref{Theorem3}  is proved.
 \end{proof}

%\subsection{Styles for the bibliographic references}

%For the bibliographical references, please see the examples of the different styles for books, papers, proceedings, etc.  included in this template. You can cite the references by means of the usual \LaTeX\ command: in the format \cite{ref01} or \cite{ref01,ref02}.

%\acknowledgements{Acknowledgements}

%It is also available this section type declaration for the acknowledgements. Please, %use it at your convenience adjusting the title of the section if you need it.

%%%%%%%%%%%%%%%%%%%%%%%%%%%%%%%%%%%%%%%%%%%%%%%%%%%%%%%%%%%%%%%%%%%%%%%%%%%%%%%%%%%%%%
%%%%%%%%%%%%%%%%%%%%%%%%%%%%%%%%%%%%%%%%%%%%%%%%%%%%%%%%%%%%%%%%%%%%%%%%%%%%%%%%%%%%%%
%       References
%%%%%%%%%%%%%%%%%%%%%%%%%%%%%%%%%%%%%%%%%%%%%%%%%%%%%%%%%%%%%%%%%%%%%%%%%%%%%%%%%%%%%%
%%%%%%%%%%%%%%%%%%%%%%%%%%%%%%%%%%%%%%%%%%%%%%%%%%%%%%%%%%%%%%%%%%%%%%%%%%%%%%%%%%%%%%

\end{document}